\documentclass[letterpaper,12pt]{amsart}

\title[HGS on degree $pq$ separable extensions]{Hopf-Galois structures on separable field extensions of degree $pq$}
\author{Andrew Darlington}
\date{\today}
\address{Department of Mathematics and Statistics, Faculty of Environment, Science and Economy, University of Exeter, Exeter EX4 QFU. UK.}
\email{ad788@exeter.ac.uk}

\usepackage[utf8]{inputenc}
\usepackage{amsmath}
\usepackage{amsfonts}
\usepackage{amssymb}
\usepackage{amsthm}
\usepackage{pifont} 
\usepackage{enumerate} 
\usepackage{enumitem}
\usepackage{wasysym} 
\usepackage{mathrsfs}
\usepackage{xfrac} 
\usepackage{array}
\usepackage{siunitx}
\usepackage{mathtools}
\usepackage{framed}
\usepackage{pifont}
\usepackage[utf8]{inputenc}
\usepackage[english]{babel}
\usepackage[utf8]{inputenc}
\usepackage{tikz,tkz-euclide}
\usepackage{xcolor,graphicx}
\usepackage{faktor}
\usepackage{xfrac}
\usepackage{float}
\usepackage{rotating}
\usepackage{xcolor}

\makeatletter
\def\bign#1{\mathclose{\hbox{$\left#1\vbox to8.5\p@{}\right.\n@space$}}\mathopen{}}
\def\Bign#1{\mathclose{\hbox{$\left#1\vbox to11.5\p@{}\right.\n@space$}}\mathopen{}}

\newtheorem{theorem}{Theorem}[section]
\newtheorem{proposition}[theorem]{Proposition}
\newtheorem{lemma}[theorem]{Lemma}
\newtheorem{remark}[theorem]{Remark}
\newtheorem{corollary}[theorem]{Corollary}

\newcommand{\e}{\mathbf{e}}
\newcommand{\Gal}{\mathrm{Gal}}
\newcommand{\Hol}{\mathrm{Hol}}
\newcommand{\Aut}{\mathrm{Aut}}

\newcommand{\cblack}{\color{black}}

\newcolumntype{M}[1]{>{\centering\arraybackslash}m{#1}}
\newcolumntype{N}{@{}m{0pt}@{}}

\begin{document}
	\maketitle
	\bibliographystyle{amsalpha}

\begin{abstract}
	In 2020, Alabdali and Byott described the Hopf-Galois structures arising on Galois field extensions of squarefree degree. Extending to squarefree separable, but not necessarily normal, extensions $L/K$ is a natural next step. One must consider now the interplay between two Galois groups $G=\Gal(E/K)$ and $G'=\Gal(E/L)$, where $E$ is the Galois closure of $L/K$. In this paper, we give a characterisation and enumeration of the Hopf-Galois structures arising on separable extensions of degree $pq$ where $p$ and $q$ are distinct odd primes. This work includes the results of Byott and Martin-Lyons who do likewise for the special case that $p=2q+1$.
\end{abstract}

\section{Introduction}\label{intro}
The concepts of Hopf-Galois theory were made explicit by Chase and Sweedler \cite{CS69} in 1969. Nearly twenty years later, Greither and Pareigis showed in \cite{GP87} that the problem of finding Hopf-Galois structures on separable field extensions could be phrased and approached purely in group-theoretic terms. In short, for a given separable extension $L/K$ of fields with Galois closure $E/K$, let $G=\Gal(E/K)$, $G'=\Gal(E/L)$, let $X=G/G'$ be the set of left cosets of $G'$ in $G$, and let $\lambda: G \rightarrow \text{Perm}(X)$ by the left translation map, where $\lambda(g)(\overline{h})=\overline{gh}$. Then each Hopf-Galois structure on $L/K$ corresponds to a regular subgroup $N$ of $\mathrm{Perm}(X)$ normalised by $\lambda(G)$. The associated Hopf algebra is given by $H=E[N]^G$, with some determined $G$-action on $E$ and $N$. The isomorphism type of $N$ is then referred to as the `type' of the Hopf-Galois structure.

Suppose $L/K$ is Galois with group $G$. Then the group (Hopf) algebra $H=K[G]$ endows $L/K$ with a Hopf-Galois structure of type $G$. If, further, $G$ is non-abelian, and $\rho: G \rightarrow \text{Perm}(G)$ is the right translation map given by $\rho(g)(h)=hg^{-1}$, then the Hopf algebras $L[\lambda(G)]^{\lambda(G)}$ and $L[{\rho(G)}]^{\lambda(G)}$ are not equal, and give distinct Hopf-Galois structures on $L/K$. See Corollary $6.11$ of \cite{Chi00} for more details.

Since \cite{GP87}, the study of Hopf-Galois structures has primarily focused on Galois extensions. In recent years, links have been discovered between Hopf-Galois structures on Galois extensions and other algebraic constructions such as skew braces (see \cite{Bac16} for braces and \cite{SV17} for skew braces), $1$-cocycles (e.g. sections $9.1$ and $9.2$ in \cite{Chi+20}), and more. How these structures arise and behave in connection to these other objects is its own very interesting question.

However, studying Hopf-Galois structures is interesting in its own right. This is evident, for example, in the study of Galois module theory, where one can now talk about the module structure of the ring of integers of a separable (but not necessarily normal) extension $L/K$ (see \cite{GM23}, for example). Another area of interest is that of the Galois correspondence. In the classical case, this gives a bijective correspondence between subgroups of the Galois group of a Galois extension $L/K$ and intermediate fields $K \leq F \leq L$. Hopf-Galois theory allows one to extend this notion, and arrive at the \textit{Hopf-Galois correspondence}. For a given Hopf-Galois structure on a field extension $L/K$ with Hopf algebra $H$, there is a correspondence from sub-Hopf algebras of $H$ and intermediate field extensions $K \leq F \leq L$. In contrast to the classical case, although this correspondence is always injective, it is not always surjective.

The study of the Hopf-Galois theory for separable but not necessarily non-normal extensions is the focus of this paper.

Many other papers have been published in this area. For $p$ a prime, \cite{CS20}, for example gives the classification of Hopf-Galois structures on separable field extensions of degree $2p$ and $p^2$; \cite{Kohl98} looks at order $p^n$ radical extensions; and, more generally, \cite{CRV16} studies this extended notion of the Galois correspondence for separable extensions in more detail. Recently, the connection between Hopf-Galois structures on Galois extensions of fields and skew braces has been extended to a connection between Hopf-Galois structures on non-normal, separable extensions of fields and skew bracoids (see \cite{MLT23}), giving an additional dimension of interest in the classification results in this paper.


In 2022, Byott and Martin-Lyons in \cite{BML21} looked at separable extensions of degree $pq$ with $p,q$ odd primes and $p=2q+1$. Such a $q$ is called a Sophie Germain prime, with $p$ being the associated `safe' prime. The paper \cite{BML21} is a first step to extending the squarefree results of \cite{AB20} to separable extensions, and also retrieves the result of \cite{Byo04} in the case where the extension is Galois. In this situation, $N$ (as above) has order $pq$, but $G=\Gal(E/K)$ may instead be a much larger group (potentially of non-squarefree order), with $pq \mid |G|$. The problem now is that no classification of such groups currently exists.

The idea presented in \cite{BML21}, which will be followed in this paper, is thus to instead consider each abstract group $N$ of order $pq$ in turn, and compute the transitive subgroups $G$ of $\Hol(N)$. Let $N_1,N_2$ be two (not necessarily distinct) abstract groups of order $pq$ with $G_1$ a transitive subgroup of $\Hol(N_1)$ and $G_2$ a transitive subgroup of $\Hol(N_2)$. This then gives the task of deciding when $G_1$ and $G_2$ are isomorphic as \emph{permutation} groups. That is, that there exists an isomorphism $\phi$ between them with the additional property that $\phi(\text{Stab}_{G_1}(1_{N_1}))=\text{Stab}_{G_2}(1_{N_2})$. If this is the case, then $G_1 \cong G_2$ corresponds to some field extension $L/K$ which will admit Hopf-Galois structures of both types $N_1$ and $N_2$.

In this paper, we show that this approach is feasible in the more general case where $p$ and $q$ are now arbitrary distinct odd primes, and we obtain a complete classification in such cases. With the results of \cite{CS20} for $p^2$ and $2p$, this paper completes the classification of Hopf-Galois structures on separable extensions whose degree is a product of two primes. Our results specialise to those in \cite{BML21} -- see remarks \ref{cyclic_BML}, \ref{metab_BML} and \ref{both_BML} that will appear later in Section \ref{general_pq} below.

\section{Preliminaries}\label{prelim}
 Let $L/K$ be a field extension and let $H$ be a $K$-Hopf algebra acting on $L$ with action $\cdot$. We say that $L$ becomes an $H$-module algebra if $\Delta(h) \cdot (x \otimes y) = \sum_{(h)}(h_{(1)} \cdot x) \otimes (h_{(2)} \cdot y)$ for all $h \in H$ and $x,y \in L$. Here $\Delta (h)=\sum_{(h)}h_{(1)}\otimes h_{(2)}$ is written in Sweedler notation. In such a situation, we say that $H$ gives a \textit{Hopf-Galois structure} on $L/K$ if the $K$-linear map $\theta: L \otimes H \rightarrow \text{End}(L)$ given by $\theta (x \otimes h)(y)=x(h \cdot y)$ is bijective. If $L/K$ is separable, recall that we denote by $E$ its Galois closure, by $G$ the Galois group of $E/K$, by $G'$ the Galois group of $E/L$, and by $X$ the left coset space of $G'$ in $G$. We say that $L/K$ is \textit{almost classically Galois} if $G'$ has a normal complement $C$ in $G$. Equivalently, there is a regular subgroup $C$ in $\text{Perm}(X)$ normalised by $G$ and contained in $G$. This class of extensions is particularly well-understood and they provide a family for which the Hopf-Galois correspondence is bijective, see Theorem 5.2 of \cite{GP87}.
 
 Take $L/K$ to be separable, with $E,G,G'$ and $X$ as defined above. We may identify the group $G$ as a subgroup of $\text{Perm}(X)$ via its image in the left-translation map $\lambda:G \rightarrow \text{Perm}(X)$, $\lambda(g)(hG')=(gh)G'$. This allows us to realise $G$ as a (transitive) permutation group of degree $[L:K]$. We see that $G$ acts transitively on the $K$-linear embeddings of $L$ into $E$, and the stabiliser of the inclusion $L \hookrightarrow E$ is $G'$.

Subsequently, a new method of finding Hopf-Galois structures was made explicit by Byott in \cite{Byo96}. Instead of looking at regular embeddings of groups in $\text{Perm}(X)$ as above, one may fix an abstract group $N$ of order $n=|L:K|$ and look at transitive subgroups $G$ of $\Hol(N) \cong N \rtimes \Aut(N)$, the \emph{holomorph} of $N$. This has the advantage of being less computationally expensive than searching directly within $\text{Perm}(X)$ as $\Hol(N)$ is, in general, a much smaller group. The paper also gives a counting formula we can use:
\begin{lemma}[\cite{Byo96}]\label{Byott_num_HGS}
	Let $G,G',N$ be as above, let $e(G,N)$ be the number of Hopf-Galois structures of type $N$ which realise $G$, and $e'(G,N)$ the number of subgroups $M$ of $\emph{Hol}(N)$ which are transitive on $N$ and isomorphic to $G$ via an isomorphism taking the stabiliser $M'$ of $1_N$ in $M$ to $G'$. Then
	\[e(G,N)=\frac{|\emph{Aut}(G,G')|}{|\emph{Aut}(N)|}e'(G,N),\]
	where
	\[\Aut(G,G')=\left\{\theta\in\Aut(G) \mid \theta(G')=G'\right\},\]
	the group of automorphisms $\theta$ of $G$ such that $\theta$ fixes the identity coset $1_GG'$ of $X=G/G'$.
\end{lemma}
The following two propositions, recalled from \cite{BML21}, Propositions 2.2 and 2.3 respectively, will be useful for our later computations:
\begin{proposition}\label{ab-aut}
	Let $N$ be an abelian group such that $\Aut(N)$ is also abelian, and let
	$A$, $A'$ be subgroups of $\Aut(N)$. Consider the subgroups $M=N
	\rtimes A$ and $M'=N \rtimes A'$ of $\Hol(N)$. If
	there is an isomorphism $\phi: M \to M'$ with
	$\phi(N)=N$, then $M=M'$.
\end{proposition}
\begin{proposition}\label{rel-aut-char}
	Let $N$ be any group and $A$ a subgroup of $\Aut(N)$. Let $M$ be the
	subgroup $N \rtimes A$ of $Hol(N)$, with $N$ a characteristic subgroup of
	$M$. Then the group
	$$ \Aut(M,A) := \{ \theta \in \Aut(M) : \theta(A)=A\}  $$ 
	is isomorphic to the normaliser of $A$ in
	$\Aut(N)$. In particular, if $\Aut(N)$ is abelian then $\Aut(M,A)
	\cong \Aut(N)$.
\end{proposition}
We denote an element of $\Hol(N)$ by $[\eta,\alpha]$ where $\eta \in N$, $\alpha \in \Aut(N)$; for two elements $[\eta,\alpha],[\mu,\beta]\in \Hol(N)$. Their product is
\[[\eta,\alpha][\mu,\beta]=[\eta\alpha(\mu),\alpha\beta].\]
We write $\eta$ for $[\eta,\text{id}_N]$ and $\alpha$ for $[1_N,\alpha]$. We also denote by $\mathbb{Z}_m$ the ring of integers modulo $m$.

\section{Hopf-Galois Structures for General $pq$}\label{general_pq}
We now consider Hopf-Galois structures on separable field extensions $L/K$ of degree $pq$ with $p>q$ arbitrary distinct odd primes. Byott in \cite{Byo96} showed that if $p \not\equiv 1 \mod q$, then $L/K$ admits a unique Hopf-Galois structure, and in that case such a structure is almost classically Galois of cyclic type. Thus, we may assume for the rest of the paper that $p \equiv 1 \bmod{q}$. We can then assume the following prime factorisations for $p-1$ and $q-1$ respectively:
\begin{align*}
	&p-1=q^{e_0}\ell_1^{e_1} \cdots \ell_m^{e_m},\\
	&q-1=\ell_1^{f_1} \cdots \ell_m^{f_m},
\end{align*}
where $e_0 > 0$, $e_i\geq$, $f_i \geq 0$, and $\text{max}\left\{e_i,f_i\right\}> 0$ for $1\leq i \leq m$. Thus we take into account the fact that $p-1$ and $q-1$ might share common factors, and that $q \mid p-1$.

There are two abstract groups $N$ of order $pq$, namely the cyclic group $C_{pq}$, and the metabelian group $C_p \rtimes C_q$ where $C_q$ acts on $C_p$ in some faithful way. We therefore split our work into looking at these two cases. In each case, we will compute all transitive subgroups $G$ of $\Hol(N)$. Thus $L/K$, corresponding to $G\cong \Gal(E/K)$ and $G' \cong \Gal(E/L) \cong \mathrm{Stab}_G(1_N)$, will admit a Hopf-Galois structure of the corresponding type.

\subsection{Cyclic case}
Let $N$ be the cyclic group of order $pq$. The idea in this section is to divide the discussion with respect to each prime appearing in the factorisations. We therefore work with the presentation
\[N = \langle \sigma,\tau |\sigma^p=\tau^q=1,\sigma\tau=\tau\sigma \rangle.\]
We thus have
\[\Aut(N) \cong \Aut(\langle \sigma \rangle) \times \Aut(\langle \tau \rangle)\]
where the factors are cyclic of order $p-1$ and $q-1$ respectively. Then $\Aut(N)$ is generated by the following elements:
	\begin{align*}
	&\alpha \in \Aut(\langle \sigma \rangle) \text{ such that ord}(\alpha)=q^{e_0},\\
	&\alpha_i \in \Aut(\langle \sigma \rangle) \text{ such that ord}(\alpha_i)=\ell_i^{e_i},\\
	&\beta_i \in \Aut(\langle \tau \rangle) \text{ such that ord}(\beta_i)=\ell_i^{f_i},
\end{align*}
where $1\leq i \leq m$. We therefore obtain the decomposition
\[\Aut(N) \cong \langle \alpha \rangle \times \langle \alpha_1,\beta_1 \rangle \times \cdots \times \langle \alpha_m,\beta_m \rangle,\]
where the factors have coprime orders. In particular, $\Aut(N)$ is abelian.

\begin{remark}\label{cyclic_subs}
	The subgroups of $\langle \alpha \rangle$ are of the form \[\left\langle \alpha^{q^{e_0-c}} \right\rangle\]
	for $0 \leq c \leq e_0$; there are $e_0+1$ of these.
\end{remark}
For the following proposition, we make a slight abuse of notation by referring to $\alpha_i$ and $\beta_i$ as simply $\alpha$ and $\beta$ respectively, to $\ell_i$ as $\ell$, and $e_i$ and $f_i$ as simply $e$ and $f$ respectively. This is to avoid messy notation.
\begin{proposition}\label{cyclic_subgroups}
	The subgroups of $\langle \alpha , \beta \rangle$ are as follows:\\
	\[\begin{array}{lll}
		\emph{(i)} \left\langle \alpha^{\ell^{e-s}},\beta^{\ell^{f-r_2}} \right\rangle, & 0 \leq s \leq e, 0 \leq r_2 \leq f,\\
		\emph{(ii)} \left\langle \alpha^{n\ell^{e-s}}\beta^{\ell^{f-r_1}} \right \rangle, &1 \leq s \leq e, 1 \leq r_1 \leq f,\\
		\emph{(iii)} \left\langle \alpha^{n\ell^{e-s}}\beta^{\ell^{f-r_1}},\beta^{\ell^{f-r_2}} \right \rangle, &1 \leq r_2 < r_1 \leq f, 1 \leq r_1-r_2 < s \leq e.
	\end{array}\]
In cases \emph{(ii)} and \emph{(iii)} we have $1 \leq n < \ell^{\min\{s,r_1\}}$ with $\ell \nmid n$.
\end{proposition}
\begin{proof}
	An arbitrary subgroup of $\langle \alpha, \beta \rangle$ is of the form $\langle \alpha^{a_1}\beta^{b_1}, \cdots, \alpha^{a_l}\beta^{b_l}  \rangle$ for some $l \geq 0$. We note that we have $\langle \alpha^{a_1}\beta^{b_1}, \cdots, \alpha^{a_l}\beta^{b_l}  \rangle = \langle \alpha^a\beta^b,\beta^c \rangle$ for some $a$, $b$ and $c$; this can be seen by looking at the matrix
	\[\begin{pmatrix}
		a_1	& b_1 \\ a_2 & b_2 \\ \vdots & \vdots \\ a_l & b_l
	\end{pmatrix}\]
	where each row represents a generator of the subgroup, and observing that the problem can be translated to finding a row echelon form of this matrix using elementary row operations. To do this, we can first assume that $a_1=\gcd(a_1, \cdots , a_l)$, and then $a_2= \cdots = a_l = 0$, redefining the $b_i$ as needed. Then we can assume that $b_2 = \gcd(b_2, \cdots , b_l)$, and then $b_3= \cdots = b_l =0$. We therefore obtain
	\[\begin{pmatrix}
		a_1	& b_1 \\ 0 & b_2 \\ 0 & 0 \\ \vdots & \vdots \\ 0 & 0
	\end{pmatrix}.\]
	We rewrite this as
	\[\begin{pmatrix}
		a	& b \\ 0 & c
	\end{pmatrix}.\]
	This precisely corresponds to the situation $\langle \alpha^a\beta^b,\beta^c \rangle$ we have above.
	
	We now write
	\begin{align*}
		&a=n\ell^{e-s},\\
		&b=m_1\ell^{f-r_1},\\
		&c=m_2\ell^{f-r_2},
	\end{align*}
	with $\ell \nmid nm_1m_2$. By replacing $b$ and $c$ with suitable powers, we may further assume, without loss of generality, that $m_1=m_2=1$.
	
	If $r_1=0$ or $s=0$, we get groups of type (i). Note that we may take a suitable power of $\alpha^a$ to get $n=1$ without loss of generality. Now assume that $s,r_1 \geq 1$, and now without loss of generality, we may assume that $r_1>r_2$ as otherwise this again corresponds to a group of type (i).
	
	If $r_2=0$, this corresponds to the following matrix:
	\[\begin{pmatrix}
		n\ell^{e-s}	& \ell^{f-r_1} \\ 0 & 0
	\end{pmatrix}.\]
	We ask which multiples $m$ of the top row of the above matrix correspond to generating the same group. That is, which multiples change $n$, but don't change $s,r_1$, and leave the coefficient of $\ell^{f-r_1}$ as $1$? To this end, we see that $m \equiv 1 \pmod{\ell^{r_1}}$. We further note that $m$ is defined modulo $\ell^{\min\{s,r_1\}}$, and so we have two cases:
	\begin{enumerate}[label=(\Roman*)]
		\item $s \leq r_1$,
		\item  $s > r_1$.
	\end{enumerate}
	Suppose first that (I) $s \leq r_1$. Then $m=1$, and so each choice of $n$ with $1 \leq n < \ell^s$ with $\ell \nmid n$ gives a distinct group. We therefore have $\varphi(\ell^s)$ such choices for $n$.
	
	Now suppose that (II) $s > r_1$. Then $m=1+x\ell^{r_1}$ for integers $x$ such that $m<\ell^s$. So each choice of $n$ with $1 \leq n < \ell^{r_1}$ with $\ell\nmid n$ gives a distinct group, giving $\varphi(\ell^{r_1})$ choices for $n$ modulo $\ell^{\min\{s,r_1\}}=\ell^{\min\{s_1,r\}}$.
	
	We may therefore combine cases (I) and (II) by saying that we have $\varphi(\ell^{\min\{s,r_1\}})$ choices for $n$. This then gives us the parameters for generating groups of type (ii).
	
	Suppose now then that $s \geq 1$ and $r_1>r_2>0$. This corresponds to the following matrix:
	\[\begin{pmatrix}
		n\ell^{e-s}	& \ell^{f-r_1} \\ 0 & \ell^{f-r_2}
	\end{pmatrix}.\]
	Note that we again have that $1 \leq n < \ell^{\min\{s,r_1\}}$ with $\ell \nmid n$, and so again have $\varphi(\ell^{\min\{s,r_1\}})$ choices for $n$. We must also have that no power of the first generator is equal to the second (otherwise obtaining a group of type (ii)). To this end, we must have $r_1-r_2<s$. We thus obtain $0<r_1-r_2<s$, noting that any choice of $r_2$ with this restriction does not affect the choices for $n$.
	
\end{proof}
For the following, we again abuse notation by referring to $\ell_i$ as $\ell$.
\begin{proposition}\label{subgroup_counts}
	Let $M:=\min\{e_i,f_i\}$. Then for each $1\leq i \leq m$, there are
	\[(e_i+1)(f_i+1)+\Sigma^{(1)}_i+\Sigma^{(2)}_i\]
	subgroups of $\langle \alpha_i,\beta_i \rangle$, where 
	\[\Sigma^{(1)}_i=(e_i+f_i+1)(\ell^M-1)-2\left[M\ell^M-\frac{1-\ell^M}{1-\ell}\right], \text{ and}\]
		\begin{align*}
			&\Sigma^{(2)}_i=(1-2M)(M-1)\ell^M-2\\
			&+\frac{1}{1-\ell}\left[\frac{4(1-\ell^M)}{1-\ell}-(2+3\ell+(4M-5)\ell^M)+(f_i+e_i)((1-M)\ell^{M+1}+M\ell^M-\ell)\right]
		\end{align*}
	give a count for the number of subgroups of types (ii) and (iii) respectively. 
\end{proposition}
\begin{proof}
	For convenience, we again drop the subscript $i$ throughout this proof. We count the subgroups in Proposition \ref{cyclic_subgroups}. In the following calculations, we adopt the convention that a sum of zero terms is $0$. It is clear to see that there are $(e+1)(f+1)$ subgroups of type (i). There are
	\[\sum_{r_1=1}^f\sum_{s=1}^e\varphi\left(\ell^{\min\{s,r_1\}}\right)\]
	subgroups of type (ii). This may be written as
	\[\sum_{r=1}^M(e+f-(2r-1))\varphi(\ell^r).\]
	The expression can be easily shown to evaluate to
	\[(e+f+1)(\ell^M-1)-2\left[M\ell^M-\frac{1-\ell^M}{1-\ell}\right].\]
		Counting subgroups of type (iii) requires us to make several ordered choices on the parameters $s,a,r_1$ and $n$, where $a:=r_1-r_2$. To begin with, we start by choosing a value for $2 \leq s \leq e$; given our choice of $s$, we may make a choice for $a$ such that $1 \leq a \leq s-1$. We are then restricted to choose an $r_1$ such that $1+a \leq r_1 \leq f$, and finally we have $\varphi(\ell^{\min\{s,r_1\}})$ choices for $n$. There are therefore
	\[\sum_{s=2}^e\sum_{a=1}^{s-1}\sum_{r_1=1+a}^f\varphi\left(\ell^{\min\{s,r_1\}}\right)\]
	subgroups of type (iii). This can be rewritten as
	\[\sum_{r=2}^M\left[(r-1)(f+e)-(r-1)^2-(r-1)r\right]\varphi(\ell^r).\]
	This sum also telescopes, and evaluates to
	\begin{align*}
		&(1-2M)(M-1)\ell^M-2\\
		&+\frac{1}{1-\ell}\left[\frac{4(1-\ell^M)}{1-\ell}-(2+3\ell+(4M-5)\ell^M)+(f+e)((1-M)\ell^{M+1}+M\ell^M-\ell)\right].
	\end{align*}
\end{proof}
As we would expect, one observes that the counts given in Proposition \ref{subgroup_counts} are all symmetric in $e$ and $f$.

We note that $H=\langle \sigma, \tau, \alpha \rangle$ is the unique subgroup of $\Hol(N)$ of order $pq^{e_0+1}$ coprime to its index $\ell_1^{e_1+f_1}\cdots\ell_m^{e_m+f_m}$. Therefore, as any transitive subgroup $M$ has order divisible by $pq$, it follows that $M \cap H$ must be transitive on $N$.
The subgroups of order divisible by $pq$ in $H$ are:
\begin{align*}
	& N \rtimes \left\langle \alpha^{q^{e_0-c}} \right \rangle, 0 \leq c \leq e_0,\\
	& J_{t,c}:=\left\langle \sigma,\left[\tau,\alpha^{tq^{e_0-c}}\right] \right\rangle, 1 \leq c \leq e_0, t \in \mathbb{Z}_{q^{e_0}}^{\times},\text{ and }\\
	& \left\langle \sigma, \alpha^{q^{e_0-c}} \right\rangle, 1 \leq c \leq e_0.
\end{align*}
It is clear that those groups containing $N$ or of the form $J_{t,c}$ are transitive on $N$. No group of the last type is transitive.

Now $N$ is normal in $\Hol(N)$, so can be extended by any subgroup of $\Aut(N)$ to give a transitive subgroup $M$. The normaliser of $J_{t,c}$ in $\Aut(N)$ is $\Aut(\langle \sigma \rangle) \cong \langle \alpha,\alpha_1, \cdots , \alpha_m \rangle$ since if $\phi \in \Aut(N)$ and $\phi(\tau)\neq \tau$, we have \[\phi [\tau,\alpha^{tq^{e_0-c}}]\phi^{-1}=[\phi(\tau),\alpha^{tq^{e_0-c}}] \notin J_{t,c}.\]
Hence if $M$ is a transitive subgroup containing $J_{t,c}$ but not $N$, then $M$ is $J_{t,c}$ extended by some subgroup of $\Aut(\langle \sigma \rangle)$. Note that $J_{t,c} \rtimes \langle \alpha^{q^{e_0-c'}} \rangle$ contains $N$ for $c'\geq c$, and is equal to $J_{t,c}$ if $c' < c$, so we take account of this below to ensure there are no repetitions in our list of transitive subgroups. The transitive subgroups of $\Hol(N)$ are shown in Table \ref{cyclic_trans_subgroups}.
\begin{table}[H]
	\setlength{\extrarowheight}{3mm}
	\begin{tabular}{|c|c|c|}
		\hline
		Key					&	Parameters				&	Group\\
		\hline
		(A)\label{N}		&	$0 \leq c \leq e_0$		&	$N \rtimes \left \langle \alpha^{q^{e_0-c}},X_i \bigm | 1\leq i \leq m \right \rangle$\\[5pt]
		\hline
		(B)\label{J_{t,c}}	&	$0 \leq c_i \leq e_0$	&	$J_{t,c} \rtimes \left \langle \alpha_i^{\ell_i^{e_i-c_i}} \bigm | 1 \leq i \leq m \right \rangle$\\[7.5pt]
		\hline
	\end{tabular}
\caption{Transitive subgroups for $N$ cyclic}
\label{cyclic_trans_subgroups}
\end{table}
	Here $X_i$ is one of the subgroups of $\langle \alpha_i , \beta_i \rangle$ outlined in Proposition \ref{cyclic_subgroups}. For groups of type (B), we have $1\leq c \leq e_0$, and in all other cases we have $0 \leq c_i \leq e_i$. In total, there are:
\[(1+e_0)\prod_{1\leq i \leq m}\left[(e_i+1)(f_i+1)+\Sigma^{(1)}_i+\Sigma^{(2)}_i\right]\]
groups of type (A).

For a given choice of $X_i$ and $c$, these groups have order $pq^{c+1}d$ where $d$ is some divisor of $(p-1)(q-1)$ coprime to $q$ coming from the orders of the $\alpha_i$ and $\beta_i$. For some factorisation $d=d_1d_2d_3$, they have abstract isomorphism type
\[ ((C_p \rtimes C_{q^cd_1}) \times (C_q \rtimes C_{d_2})) \rtimes C_{d_3},\]
where, slightly abusing notation, $C_{d_1} \leq \Aut(\langle\sigma\rangle)$, $C_{d_2} \leq \Aut(\langle\tau\rangle)$, and $C_{d_3}$ is generated by elements of the form $\alpha_i^x\beta_i^y$ for $1 \leq i \leq m$.

For fixed $t$ and $c$, the number of groups of type (B) is
\[\prod_{1 \leq i \leq m}(e_i+1).\]
For a given choice of the $c_i$, and for any $t$, such a group has order
\[pq^{c}\prod_{1 \leq i \leq m}\ell_i^{c_i}\]
note that there are $\varphi(q^c)$ such groups of this order. They are abstractly isomorphic to
\[ C_p \rtimes C_{q^{c}d}\]
where $d=\ell_1^{c_1}\cdots \ell_m^{c_m}$ is some divisor of $(p-1)$ coprime to $q$.

The stabiliser of $1_N$ in the above subgroups for which $c=1$ has normal complement (either $N$ or $J_{t,1}$), so all the corresponding field extensions to those groups are almost classically Galois. For $c>1$, $J'_{t,c}:=\mathrm{Stab}_{J_{t,c}}(1_N)=\langle \alpha^{q^{e_0-(c-1)}} \rangle$. A normal complement to this subgroup in $J_{t,c}$ would have order $pq$, but the only subgroup of $J_{t,c}$ of order $pq$ is $\langle \sigma,\alpha^{q^{e_0-1}} \rangle$, which has nontrivial intersection with $J'_{t,c}$. Thus any transitive subgroup of type (B) for $c>1$ does not contain a regular normal subgroup, and so the corresponding extensions fail to be almost classically Galois.

\begin{lemma}\label{cyclic_isoms}
	Fix $1\leq c \leq e_0$. Then for each choice of subgroup $A\leq\langle \alpha_1,\cdots,\alpha_m\rangle$, the $\varphi(q^{c})$ groups given by $J_{t,c} \rtimes A$ for $1 \leq t <q^c$ with $q \nmid t$, are isomorphic. Between any two transitive subgroups of type \emph{(A)} or \emph{(B)}, there are no other isomorphisms, either as permutation subgroups of $\Hol(N)$, or even as abstract groups.
\end{lemma}
\begin{proof}
	Let $\phi \in \Aut(N)$ such that $\phi(\tau)=\tau^t$, then
	\[\phi [\tau,\alpha^{tq^{e_0-c}}]\phi^{-1}= [\phi(\tau),\alpha^{tq^{e_0-c}}]= [\tau,\alpha^{q^{e_0-c}}]^t.\]
	Thus conjugation by $\phi$ gives an isomorphism between $J_{t,c}$ and $J_{1,c}$, which extends to an isomorphism between $J_{t,c} \rtimes A$ and $J_{1,c} \rtimes A$ with $A<\langle \alpha_1,\cdots,\alpha_m \rangle$. Note that as conjugation by $\phi$ fixes any automorphism, and hence fixes the stabiliser of $1_N$, it also shows that they are isomorphic as permutation groups.
	Next, note that two groups $J_{1,c}\rtimes A$ and $J_{1,c'}\rtimes A'$ with $A,A'<\langle \alpha_1,\cdots,\alpha_m \rangle$ are isomorphic if and only if $c=c'$ and $A=A'$. This is clear when noting that choosing different values for $c$ or generators for $A$ gives a group of a different order.
	Finally, by Proposition \ref{ab-aut}, no two groups of type (A) are isomorphic, nor can they be isomorphic to any group containing $J_{1,c}$ for any $c$ and $A<\langle \alpha_1, \cdots, \alpha_m \rangle$ because the groups $J_{1,c} \rtimes A$ do not contain an abelian subgroup of order $pq$.
\end{proof}

\begin{lemma}\label{cyclic_HGS_isom}
	The numbers of Hopf-Galois structures per isomorphism class of transitive subgroups of $\Hol(N)$ are as in Table \ref{cyclic-HGS}.
\end{lemma}
	\begin{table}
		\setlength{\extrarowheight}{1.5mm}
		\bigskip
		
		\centering
		\scalebox{1}{
			\begin{tabular}{|c|c|c|c|}
				\hline
				Parameters 			& Group 				& \#HGS per isom. class & Comments\\
				\hline
				$A\leq \Aut(N)$		& $N \rtimes A$			& $1$					& each in unique isom. class\\
				\hline
				$1\leq t \leq q-1$	& $J_{t,1}$				& $p$					& single isom. class\\
				\hline
				$c>1$,  $t \in \mathbb{Z}_{q^c}$ or 			&						&						& unique isom. class for\\
				$\{1\}\neq A \leq \Aut(\langle\sigma\rangle)$  & $J_{t,c}\rtimes A$ & $q^{c-1}$ & fixed $A$ and $c$ \\
				with $\alpha \notin A$ &						&						&\\
				\hline
		\end{tabular}
	}
		\vskip5mm
		\caption{Hopf-Galois Structures of Cyclic Type}
		\label{cyclic-HGS}
	\end{table}
\begin{proof}
	The first two rows follow from Lemma $4.3$ of \cite{BML21}. So now suppose $M:=J_{t,c}\rtimes A$ and suppose that either $c>1$ or $A$ is nontrivial. Without loss of generality, we may assume that $t=1$. Suppose explicitly that $A=\langle \alpha_1^{b_1},\cdots,\alpha_m^{b_m}\rangle$, such that $\alpha_i(\sigma)=\sigma^{a_{\alpha_i}}$. Let $M':=\left\langle \alpha^{q^{e_0-(c-1)}} \right\rangle \times A$ denote the stabiliser of $1_N$ in $M$. We first wish to compute $|\Aut(M,M')|$, and so consider an element $\phi \in \Aut(M)$. We have the following
	\begin{align}
		&\phi(\sigma)=\sigma^x \text{ for some } 1 \leq x \leq p-1, \label{phi_sigma}\\
		&\phi\left(\left[\tau,\alpha^{q^{e_0-c}}\right]\right)=\sigma^y\left[\tau,\alpha^{q^{e_0-c}}\right] \text{ for some } 0\leq y \leq p-1, \text{ and} \label{phi_taualpha}\\
		&\phi(\alpha_i^{y_i})=[\sigma^{x_i},\alpha_i^{y_i}] \text{ for some } 0\leq x_i < p-1. \label{phi_alphai}
	\end{align}
	Given that the generator $\left[\tau,\alpha^{q^{e_0-c}}\right]$ commutes with any $\alpha_i^{y_i}$, we must have
	\begin{equation}\label{alpha_commute}
		y(1-a_{\alpha_i}^{y_i})=x_i(1-a_{\alpha}^{q^{e_0-c}}) \;\; \text{ for all } 1\leq i \leq m.
	\end{equation}	
	Now, if $\phi$ fixes $M'$, then $x_i=0$ in (\ref{phi_alphai}), so $\phi$ acts as identity on $A$. By (\ref{phi_taualpha}), the image of $\alpha^{q^{e_0-(c-1)}}$ under $\phi$ is
	\[\left(\sigma^y\left[\tau,\alpha^{q^{e_0-c}}\right]\right)^q.\]
	When $c>1$, this is in $M'$ if and only if $y=0$. Note that if $c=1$ with $A$ nontrivial, then (\ref{alpha_commute}) with $x_i=0$ again forces $y=0$. There are therefore $p-1$ choices for $\phi$.
	
	Now, as there are $\varphi(q^c)=q^{c-1}(q-1)$ choices for $t$ in the isomorphism class of $M$, there are therefore
	\[\varphi(q^c)\frac{|\Aut(M,M')|}{|\Aut(N)|}=\varphi(q^c)\frac{p-1}{(p-1)(q-1)}=q^{c-1}\]
	Hopf-Galois structures of type $C_{pq}$ with group $M$.
\end{proof}
Summarising, we have:
\begin{theorem}\label{cyclic_HGS_total}
	There are
	\[(1+e_0)\prod_{1\leq i \leq m}\left[(e_i+1)(f_i+1)+\Sigma_i^{(1)}+\Sigma_i^{(2)}\right] + e_0\prod_{1 \leq i \leq m}(e_i+1)\]
	isomorphism types of permutation groups $G$ of degree $pq$ which are realised by a Hopf-Galois structure of cyclic type.
	
	These include the two regular groups, i.e. the
	cyclic and non-abelian groups of order $pq$ (for which the corresponding
	Galois extensions have $1$ and $p$ Hopf-Galois structures of cyclic type
	respectively).
	
	There are $(e_0-1)\prod_{1 \leq i \leq m}(e_i+1)$ groups for which the corresponding extensions fail to be almost classically Galois; these correspond to groups of type \emph{(B)} with $c>1$. For each of these, there are $q^{c-1}$ Hopf-Galois structures of cyclic type.
	
	For all the remaining groups $G$, any field extension $L/K$	realising $G$ is almost classically Galois and admits a unique Hopf-Galois structure of cyclic type.
\end{theorem}
Note the first and second summands count the number of isomorphism classes of groups of type (A) and (B) respectively.

\begin{remark}\label{cyclic_BML}
	\emph{In the setting of \cite{BML21}, we have $p-1=2q$ with $q-1=2^rs$ where $s$ is odd. Thus, specialising our results to this case, we take  $e_0=1$, $\ell_1=2$, $e_1=1$, $f_1=r$, $s=\ell_2^{f_2} \cdots \ell_m^{f_m}$, and $e_i=0$ for $2 \leq i \leq m$. We get $\Sigma_1^{(1)}=r$, $\Sigma_i^{(1)}=0$ for $2 \leq i \leq m$ and $\Sigma_i^{(2)}=0$ for $1 \leq i \leq m$. Therefore we recover the Cyclic case result of their paper, that there are
		\[2(2(r+1)+r)\prod_{i=2}^m(f_i+1)=(6r+4)\sigma_0(s)\]
		isomorphism types of permutation groups $G$ of degree $pq$ (with $p=2q+1$) which are realised by a Hopf-Galois structure of cyclic type. (Here, $\sigma_0(s)$ counts the number of divisors of $s$.)}
\end{remark}

\subsection{Metabelian case}
Now let $N$ be the non-abelian group of order $pq$:
\[\langle \sigma,\tau | \sigma^p=\tau^q=1,\tau\sigma=\sigma^g\tau \rangle\]
where $g$ has order $q \bmod{p}$. Let $s=(p-1)/q^{e_0}$. Then, by equation $(3.7)$ in \cite{Byo04}, $\Aut(N)$ has order $p(p-1)=pq^{e_0}s$, and is generated by automorphisms $\alpha$, $\beta$, $\epsilon$ of orders $q^{e_0}$, $s$, $p$ respectively, where:
\begin{align*}
	&\alpha(\sigma)=\sigma^{a_{\alpha}}, &&\alpha(\tau)=\tau,\\
	&\beta(\sigma)=\sigma^{a_{\beta}},	&&\beta(\tau)=\tau,\\
	&\epsilon(\sigma)=\sigma,			&&\epsilon(\tau)=\sigma\tau.
\end{align*}
We have $\text{ord}_p(a_{\alpha})=q^{e_0}$ and $\text{ord}_p(a_{\beta})=s$. Without loss of generality, we can assume that $g=a_{\alpha}^{q^{e_0-1}}$. The generators satisfy the following relations:
\[\begin{array}{lcr}
	\beta\alpha=\alpha\beta,	&\alpha\epsilon=\epsilon^{a_{\alpha}}\alpha,	&\beta\epsilon=\epsilon^{a_{\beta}}\beta.
\end{array}\]
So $\Hol(N)=\langle \sigma, \tau, \alpha, \beta, \epsilon \rangle$ is a group of order $sp^2q^{1+e_0}$. Using an idea in \cite{BML21}, we first note that $P:= \langle \sigma, \epsilon \rangle \cong C_p \times C_p$ is the unique Sylow $p$-subgroup of $\Hol(N)$, with complementary subgroup $R:= \langle \tau, \alpha, \beta \rangle \cong C_q \times C_{q^{e_0}} \times C_s$. Thus $\Hol(N)\cong P \rtimes R$. We see that the element $\sigma\epsilon^{g-1} \in P$ commutes with $\tau$:
\[\sigma\epsilon^{g-1}\tau=\sigma\epsilon^{g-1}(\tau)\epsilon^{g-1}=\sigma^g\tau\epsilon^{g-1}=\tau\sigma\epsilon^{g-1}.\]
Writing $P$ additively, we can identify it with the vector space $\mathbb{F}_p^2$, with basis vectors
\begin{align*}
	&\e_1 = \begin{pmatrix} 1\\0	\end{pmatrix},
	&&\e_2 = \begin{pmatrix} 0\\1	\end{pmatrix},
\end{align*}
corresponding to $\sigma$, $\sigma\epsilon^{g-1}$ respectively. The generators $\tau$, $\alpha$, $\beta$ of $R$ can then be identified with the matrices
\[\begin{array}{ccc}
	T=\begin{pmatrix} g & 0 \\ 0 & 1 \end{pmatrix}, &A=\begin{pmatrix} a_{\alpha} & 0 \\ 0 & a_{\alpha} \end{pmatrix}, &B=\begin{pmatrix} a_{\beta} & 0 \\ 0 & a_{\beta} \end{pmatrix}
\end{array}\]
respectively. An element of $\Hol(N)$ is thus written as $[\textbf{v},U]$ with $\textbf{v} \in \mathbb{F}_p^2$ and $U$ a matrix corresponding to an element of $R$.
We now determine conditions for a subgroup of $\Hol(N)$ to be transitive on $N$:
\begin{lemma}\label{metab_trans_conditions}
	A subgroup $M$ of $\Hol(N)$ is transitive on $N$ if and only if it satisfies the following two conditions:
	\begin{enumerate}[label=\emph{(\roman*)}]
		\item the image of $M$ under the quotient map $\Hol(N) \rightarrow R$ is one of
		\begin{itemize}
			\item $\left\langle TA^{uq^{e_0-c}},B^{s/d} \right \rangle$, $0 \leq c \leq e_0$, $d|s$, $u \in \mathbb{Z}_{q^c}^{\times}$,
			\item $\left\langle T,A^{q^{e_0-c}},B^{s/d} \right \rangle$, $1 \leq c \leq e_0$, $d|s$.
		\end{itemize}
		\item $M \cap P$ is one of $\mathbb{F}_p^2$, $\mathbb{F}_p\e_1$, $\mathbb{F}_p\e_2$, each of which is normalised by $R$.
	\end{enumerate}
\end{lemma}
\begin{proof}
	Suppose that $M$ is transitive. Then the orbit of $1_N$ under $M \cap P$ must have size $p$, so the projection of $M$ into $R$ cannot be contained in $\Aut(N)$, and must include $T$ in some way. Hence we obtain one of the groups listed in (i).
	
	If $P \nsubseteq M$ then $M \cap P$ has order $p$. Thus $M \cap P= \mathbb{F}_p\e_1$ or $\mathbb{F}_p(\lambda\e_1+\e_2)$ for some $\lambda \in \mathbb{F}_p$. As $M$ contains an element of the form $[\textbf{v},TA^aB^b]$ for some $\textbf{v}\in P$, and $a$ and $b$ in their respective ranges, in the latter case, we have
	\[TA^aB^b(\lambda\e_1+\e_2)=a_{\alpha}^aa_{\beta}^b(g\lambda\e_1+\e_2).\]
	The vector $g\lambda\e_1+\e_2$ does not lie in $M \cap P$ unless $\lambda=0$. Thus $M \cap P$ is one of the subgroups listed in (ii). It is clear that these are all normalised by $R$.
	
	Conversely, it is easy to check that if $M$ satisfies (i) and (ii), then $M$ is transitive.
\end{proof}

\begin{corollary}\label{metab_p^2q_groups}
	The transitive subgroups containing $P$ are:
	\begin{enumerate}[label=\emph{(\roman*)}]
		\item $P \rtimes \left \langle T, A^{q^{e_0-c}},B^{s/d} \right \rangle$, $0 \leq c \leq e_0$, $d|s$. These groups have order $dp^2q^{1+c}$.
		\item $P \rtimes \left \langle TA^{uq^{e_0-c}},B^{s/d} \right \rangle$, $1 \leq c \leq e_0$, $u \in \mathbb{Z}_{q^c}^{\times}$, $d|s$. These groups have order $dp^2q^c$.
	\end{enumerate}
\end{corollary}
\begin{corollary}\label{metab_notP}
	The transitive subgroups not containing $P$ are generated by the set
	\begin{enumerate}[label=\emph{(\Roman*)}]
		\item $\left\{\e_i,[\lambda\e_{3-i},T],[\mu\e_{3-i},A^{q^{e_0-c}}],[\nu\e_{3-i},B^{s/d}]\right\}$, or the set
		\item $\left\{\e_i,[\lambda\e_{3-i},TA^{uq^{e_0-c}}],[\mu\e_{3-i},B^{s/d}]\right\}$
	\end{enumerate}
for choices of $i \in \left\{1,2\right\}$, $0\leq c \leq e_0$ for \emph{(I)}, $1\leq c \leq e_0$ for \emph{(II)}, $u \in \mathbb{Z}_{q^c}^{\times}$, $d \mid s$, and $\lambda, \mu,\nu \in \mathbb{F}_p$.
\end{corollary}
We note that not all values of $\lambda, \mu$ and $\nu$ will work, so the following proposition seeks to find the necessary and sufficient conditions on these parameters.
\begin{proposition}[Transitive subgroups]
	Table \ref{metabelian-trans-subgroups} lists the transitive subgroups of $\Hol(N)$ for $N$ metabelian of order $pq$.
\end{proposition}
\begin{proof}
	The first two rows in the table are given by Corollary \ref{metab_p^2q_groups}. We now use Corollary \ref{metab_notP} and consider just the case $i=1$. The case $i=2$ is similar.
	
	Suppose first that $M$ is a transitive subgroup generated by the elements in the set given by (I), and consider the generators
	\[x:=\left[\lambda\e_2,T\right], \;\; y:=\left[\mu\e_2,A^{q^{e_0-c}}\right], \;\; z:=\left[\nu\e_2,B^{s/d}\right].\]
	We have that $x^q=q\lambda\e_2$. Therefore $\e_2 \in M$ if and only if $\lambda \neq 0$, but as $P \not\subseteq M$, we must have $\e_2 \notin M$ and so $\lambda=0$.
	Since $\mathbb{F}_p\e_2$ is an eigenspace for $R$ and intersects $P \cap M$ trivially, we require $z$ to normalise the subgroup $\langle \e_1, T, y \rangle$ of order $pq^{1+c}$, and so \[zyz^{-1}=\left[\left(\nu\left(1-a_{\alpha}^{q^{e_0-c}}\right)+a_{\beta}^{s/d}\mu\right)\e_2,A^{q^{e_0-c}}\right],\]
	which is an element of $\langle \e_1, T, y \rangle$ if and only if
	\[\nu\left(1-a_{\alpha}^{q^{e_0-c}}\right)+a_{\beta}^{s/d}\mu=\mu.\]
	That is, if and only if
	\[\left(1-a_{\alpha}^{q^{e_0-c}}\right)\nu=\left(1-a_{\beta}^{s/d}\right)\mu.\]
	We therefore obtain the third row in Table \ref{metabelian-trans-subgroups}, noting that, abusing notation, we rewrite $\lambda,\mu$ and $\nu$ to just depend on the new parameter $\lambda$.
	
	Suppose now that $M$ is a transitive subgroup generated by the elements in the set given by (II), and consider the generators
	\[x:=[\lambda\e_2,TA^{uq^{e_0-c}}] \;\; \text{ and } \;\; y:=[\mu\e_2,B^{s/d}].\]
	Recall that $c\geq 1$. We require $y$ to normalise the subgroup $\langle \e_1,x \rangle$ of order $pq^c$, and so
	\[yxy^{-1}=\left[\left(\mu\left(1-a_{\alpha}^{uq^{e_0-c}}\right)+a_{\beta}^{s/d}\lambda\right)\e_2,TA^{uq^{e_0-c}}\right],\]
	which is an element of $\langle \e_1,x \rangle$ if and only if
	\[\mu\left(1-a_{\alpha}^{uq^{e_0-c}}\right)+a_{\beta}^{s/d}\lambda=\lambda.\]
	That is, if and only if
	\[\left(1-a_{\alpha}^{uq^{e_0-c}}\right)\mu=\left(1-a_{\beta}^{s/d}\right)\lambda.\]
	We therefore obtain the fourth row in Table \ref{metabelian-trans-subgroups}, noting that, abusing notation, we rewrite $\lambda$ and $\mu$ to just depend on the new parameter $\lambda$.
	
	Similar discussions for $i=2$ give the final two rows.
\end{proof}
\begin{sidewaystable}
	\setlength{\extrarowheight}{3.5mm}
	\vskip135mm
	\bigskip
	\centering
	\scalebox{1}{
		\begin{tabular}{|c|c|}
			\hline
			Parameters	&	Group\\
			\hline
			$0\leq c \leq e_0$,  $d|s$			&	$P \rtimes \left\langle T,A^{q^{e_0-c}},B^{s/d} \right\rangle$\\[5pt]
			\hline
			$1\leq c\leq e_0$,  $d|s$, $u \in \mathbb{Z}_{q^c}^{\times}$&	$P \rtimes \left\langle TA^{uq^{e_0-c}},B^{s/d} \right\rangle$\\[5pt]
			\hline
			$0 \leq c \leq e_0,0\leq \lambda \leq p-1$,  $d|s$			&	$\left\langle \e_1,T,\left[(1-a_{\alpha}^{q^{e_0-c}})\lambda\e_2,A^{q^{e_0-c}}\right],\left[(1-a_{\beta}^{s/d})\lambda\e_2,B^{s/d}\right] \right\rangle$\\[5pt]
			\hline
			$1 \leq c \leq e_0,0\leq \lambda \leq p-1$,  $d|s$, $u \in \mathbb{Z}_{q^c}^{\times}$			&	$\left\langle \e_1, \left[(1-a_{\alpha}^{uq^{e_0-c}})\lambda\e_2,TA^{uq^{e_0-c}}\right],\left[(1-a_{\beta}^{s/d})\lambda\e_2,B^{s/d}\right] \right\rangle$\\[5pt]
			\hline
			$0 \leq c \leq e_0,0\leq \lambda \leq p-1$,  $d|s$			&	$\left\langle \e_2,\left[(1-g)\lambda\e_1,T\right],\left[(1-a_{\alpha}^{q^{e_0-c}})\lambda\e_1,A^{q^{e_0-c}}\right],\left[(1-a_{\beta}^{s/d})\lambda\e_1,B^{s/d}\right] \right\rangle$\\[5pt]
			\hline	
			$1 \leq c \leq e_0,0\leq \lambda \leq p-1$,  $d|s$, $u \in \mathbb{Z}_{q^c}^{\times}$			&	$\left\langle \e_2,\left[(1-ga_{\alpha}^{uq^{e_0-c}})\lambda\e_1,TA^{uq^{e_0-c}}\right],\left[(1-a_{\beta}^{s/d})\lambda\e_1,B^{s/d}\right] \right\rangle$\\[5pt]
			\hline
	\end{tabular}}
	\vskip5mm

	\caption{Transitive subgroups for $N$ metabelian}
	\label{metabelian-trans-subgroups}
\end{sidewaystable}

We again denote by $M'$ the stabiliser of $1_N$ in $M$. In order to count the number of Hopf-Galois structures of type $C_p\rtimes C_q$, we now move on to compute $|\Aut(M,M')|$ for each of the subgroups $M$ given in Table \ref{metabelian-trans-subgroups}. We note that $M'=M \cap \Aut(N)$, and that $\Aut(N)=\langle \textbf{f},A,B \rangle$, where $\textbf{f}=\e_1-\e_2$ is the element of $\mathbb{F}_p^2$ corresponding to $\epsilon^{1-g} \in \Aut(N)$.

\begin{proposition}[Isomorphisms]\label{metab_isoms}
	For a fixed $c,d$ and row in Table \ref{metabelian-trans-subgroups}, the groups parametrised by $0 \leq \lambda \leq p-1$ are all isomorphic as permutation groups.
	
	Furthermore, we have the following isomorphisms of permutation groups (we assume that $\lambda=0$, and all other parameters run over all appropriate values unless otherwise stated):
	\[\begin{array}{|c|c|c|c|c|}
		\hline
		\text{Key}	&	\text{Row in Table \ref{metabelian-trans-subgroups}}	& 	\text{Parameters}	&	\text{Groups}	&	\text{Order}\\
		\hline
		\text{(i)}	&	3			&	c=0					&	\langle \e_1,T,B^{s/d} \rangle	&	dpq\\
					&	4			&	c=1, u\neq -1		&	\langle \e_1,TA^{uq^{e_0-1}},B^{s/d} \rangle	&\\
					&	6			&	c=1					&	\langle \e_2,TA^{uq^{e_0-1}},B^{s/d} \rangle	&\\
		\hline
		\text{(ii)}	&	3			&	c=1, u=-1			&	\langle \e_1,TA^{-q^{e_0-1}},B^{s/d} \rangle	&	dpq\\
					&	5			&	c=0					&	\langle \e_2,T,B^{s/d} \rangle	&\\
		\hline
		\text{(iii)}&	3			&	c>0					&	\langle \e_1,T,A^{q^{e_0-c}},B^{s/d} \rangle	&	dpq^{1+c}\\
					&	5			&	c>0					&	\langle \e_2,T,A^{q^{e_0-c}},B^{s/d} \rangle	&\\
		\hline
		\text{(iv)}	&	4			&	c>1					&	\langle \e_1,TA^{uq^{e_0-c}},B^{s/d} \rangle	&	dpq^c\\
					&	6			&	c>1					&	\langle \e_2,TA^{vq^{e_0-c}},B^{s/d} \rangle	&\\
		\hline
	\end{array}\]
	Within rows three to six of Table \ref{metabelian-trans-subgroups}, there are no more isomorphisms, even as abstract groups.
\end{proposition}
\begin{proof}
	For appropriately fixed $c$ and $d$, one first sees that any group in rows three and four of Table \ref{metabelian-trans-subgroups} is conjugate, via some $\mu\e_2$, to the group in the same row with the same $c,d$ but with $\lambda=0$. Likewise for the groups in rows five and six, but via some $\mu\e_1$. This gives an isomorphism as permutation groups, and we may therefore assume that $\lambda=0$ for the rest of the paper, noting that the group represents an isomorphism class of size $p$.
	
	The isomorphisms given in the statement of the proposition should not be too difficult to see. For example, the isomorphism between $M_1:=\langle \e_1,T,B^{s/d} \rangle$ and $M_2:=\langle \e_1,TA^{uq^{e_0-1}},B^{s/d} \rangle$ with $u\neq -1$ may be as follows:
	\begin{align*}
		&\phi(\e_1)=\e_1,\\
		&\phi(T)=\left(TA^{uq^{e_0-1}}\right)^{(1+u)^{-1}},\\
		&\phi(B^{s/d})=B^{s/d}.
	\end{align*}
	This gives an isomorphism of permutation groups as $\phi$ also gives an isomorphism between $\mathrm{Stab}_{M_1}(1_N)=\langle B^{s/d} \rangle$ and $\mathrm{Stab}_{M_2}(1_N)=\langle B^{s/d} \rangle$. The isomorphism between $M_1$ and $M_3:=\langle \e_2,TA^{uq^{e_0-1}},B^{s/d} \rangle$ may be as follows:
	\begin{align*}
		&\phi(\e_1)=\e_2,\\
		&\phi(T)=\left(TA^{uq^{e_0-1}}\right)^{u^{-1}},\\
		&\phi(B^{s/d})=B^{s/d}.
	\end{align*}
	This again gives an isomorphism of permutation groups. The other isomorphisms in the proposition are computed similarly. 
	
	 Noting that every group in rows three to six of Table \ref{metabelian-trans-subgroups} appears exactly once in the proposition, we must now show that there are no further isomorphisms between them.

	 For any $d \mid s$, the groups of type (i) are not isomorphic to any of type (ii) as those of type (ii) contain an abelian group of order $pq$ but those of type (i) do not. For any $1 \leq c \leq e_0$ and $d\mid s$, neither groups of type (i) or (ii) are isomorphic to groups of type (iii) or (iv) as they have order not divisible by $q^2$. For any $1 \leq c \leq e_0$ and $d\mid s$, groups of type (iii) are not isomorphic to groups of type (iv) as those of type (iv) have a cyclic Sylow $q$-subgroup, but those of type (iii) do not. Finally, groups of type (i), (ii), (iii) and (iv) are not isomorphic to groups of the same type for different values of $c$ and $d$, again due to an order argument.	
	
\end{proof}
We now move to determining isomorphisms between groups in the first two rows of Table \ref{metabelian-trans-subgroups}, along with computing $|\Aut(M,M')|$ for each transitive subgroup $M$ of $\Hol(N)$. In light of Proposition \ref{metab_isoms}, for the groups of order not divisible by $p^2$, we now need only consider the groups in rows three and four of Table \ref{metabelian-trans-subgroups} with $\lambda=0$.

\begin{proposition}\label{metab_holgp_isoms}
	Let $M$ be a transitive group in the first row of Table \ref{metabelian-trans-subgroups}. We have Table \ref{row1}.
	
	No two groups in Table \ref{row1} are isomorphic even as abstract groups.
	
	Further, the group $P \rtimes \langle T, B^{s/d} \rangle$ is isomorphic as a permutation group to the group $P \rtimes \langle TA^{-q^{e_0-1}}, B^{s/d} \rangle$ in row $2$ of Table \ref{metabelian-trans-subgroups} with $(u,c)=(-1,1).$
	
	There are no further isomorphisms between groups in row $1$ and groups in row $2$ of Table \ref{metabelian-trans-subgroups}.
	\begin{table}
		\begin{tabular}{|c|c|c|c|}
			\hline
			Parameters	&	$M$	&	Order	&$|\Aut(M,M')|$\\
			\hline
			$c>0$ or $d>1$	& $P \rtimes \langle T,A^{q^{e_0-c}},B^{s/d} \rangle$ & 	$p^2q^{1+c}d$	&$2p(p-1)$\\
			\hline
			$d=0,d=1$				& $P \rtimes \langle T \rangle$	&	$p^2q$	& $p(p-1)$\\
			\hline	
		\end{tabular}
		\caption{$|\Aut(M,M')|$ for groups in row $1$ of Table \ref{metabelian-trans-subgroups}}
		\label{row1}
	\end{table}
\end{proposition}

\begin{proof}
	It is clear that no two groups in Table \ref{row1} are isomorphic as abstract groups because each group has a different order. An isomorphism $\phi$ from $P \rtimes \langle TA^{-q^{e_0-1}}, B^{s/d} \rangle$ to the transitive subgroup $P \rtimes \langle T, B^{s/d} \rangle$ may be given as follows:
	\begin{align*}
		&\phi(\e_i)=\e_{3-i},\\
		&\phi(TA^{-q^{e_0-1}})=T^{-1},\\
		&\phi(B^{s/d})=B^{s/d}.
	\end{align*}
	This also gives an isomorphism of permutation groups as it preserves the stabiliser, $\langle \textbf{f},B^{s/d} \rangle$, of the identity. We note further that, apart from the isomorphisms just described, there are no further isomorphisms between the groups in the first row of Table \ref{metabelian-trans-subgroups} and the groups in the second row. This is because the groups in row $1$ contain an abelian group of order $pq$, namely $\langle \e_2,T \rangle$, but the groups in row $2$ for $(u,c) \neq (-1,1)$ do not.
	
	For each such transitive group $M$, we now compute $|\Aut(M,M')|$. Recall that $\textbf{f}:=\e_1-\e_2$, and that $M'=M\cap \Aut(N)$, and so we have $M'=\langle \textbf{f},A^{q^{e_0-c}},B^{s/d} \rangle$, and suppose that $\theta \in \Aut(M,M')$. We count the number of possible images of $\theta$ on each generator of $M$, noting that we have the restrictions
	\begin{align}
		&\theta(\textbf{f})=\lambda\textbf{f}, \; 1\leq \lambda \leq p-1,\label{theta_f}\\
		&\theta(A^{q^{e_0-c}})=\left[\mu\textbf{f},A^{aq^{e_0-c}}\right], \; 0\leq \mu \leq p-1, a \in \mathbb{Z}_{q^c}^{\times},\label{theta_A}\\
		&\theta(B^{s/d})=\left[\nu\textbf{f},B^{bs/d}\right], \; 0\leq \nu \leq p-1, (b,d)=1. \label{theta_B}
	\end{align}
	As $M$ contains exactly two normal subgroups, namely $\langle\e_1\rangle$ and $\langle\e_2\rangle$, we must have that either $\theta(\e_i)=x_ie_i$ or $\theta(\e_i)=x_i\e_{3-i}$ for $i \in \{1,2\}$ and $1\leq x_i \leq p-1$. By equation (\ref{theta_f}), we obtain that $x_1=x_2$ in both cases. We write
	\begin{equation}\label{theta_T}
		\theta(T)=\left[\textbf{v},T^jA^{kq^{e_0-1}}\right]
	\end{equation}
	for some $\textbf{v} \in \mathbb{F}_p^2,1 \leq j \leq q-1,0 \leq k \leq q-1$, and subsequently determine restrictions on these parameters in the different cases.
	
	We consider separately the case that $(c,d)=(0,1)$. In this case, we have
	\[M= P \rtimes \langle T \rangle \cong C_p \times (C_p \rtimes C_q),\]	
	and $M'=\langle \textbf{f} \rangle$. From (\ref{theta_T}), we also have that $[\textbf{v},T]^q=[1,\text{id}]$, and so $\textbf{v}$ must be an element of $\mathbb{F}_p\e_1$, say $\textbf{v}=\eta\e_1$ for $0 \leq \eta \leq p-1$.
	
	First let $\theta(\e_i)=x\e_i$, then we obtain no further restrictions on the images of $\theta$, giving $p(p-1)$ possibilities. If $\theta(\e_i)=x\e_{3-i}$, then the relation $T\e_2=\e_2T$ means that there is no choice for $\eta$ that works. We therefore have $|\Aut(M,M')|=p(p-1)$.
	
	Now suppose that $(c,d) \neq (0,1)$, and suppose first that $\theta(\e_i)=xe_i$. The relations $A^{q^{e_0-c}}\e_i=a_{\alpha}^{q^{e_0-c}}\e_iA^{q^{e_0-c}}$, $B^{s/d}\e_i=a_{\beta}^{s/d}\e_iB^{s/d}$, $T\e_1=g\e_1T$ and $T\e_2=\e_2T$ give us $a=b=j=1$ and $k=0$. Recalling that $R=\langle T,A,B \rangle$ is abelian, we obtain the following relations:
	\begin{align*}
		&\left(1-a_{\alpha}^{q^{e_0-c}}\right)\textbf{v}=\mu\left(I-T\right)\textbf{f},\\
		&\left(1-a_{\beta}^{s/d}\right)\textbf{v}=\nu\left(I-T\right)\textbf{f},\text{ and }\\
		&\left(1-a_{\alpha}^{q^{e_0-c}}\right)\nu=\left(1-a_{\beta}^{s/d}\right)\mu,
	\end{align*}
	where $I$ is the identity matrix. We therefore have:
	\begin{align*}
		&\theta(\e_i)=x\e_i,	&&\theta(A^{q^{e_0-c}})=\left[\left(1-a_{\alpha}^{q^{e_0-c}}\right)\eta\textbf{f},A^{q^{e_0-c}}\right],\\	&\theta(T)=\left[\eta\e_1,T\right],	&&\theta(B^{s/d})=\left[\left(1-a_{\beta}^{s/d}\right)\eta\textbf{f},B^{s/d}\right],
	\end{align*}
	for $0 \leq \eta \leq p-1$. This gives $p(p-1)$ choices for $\theta$.
	
	Suppose now that $\theta(\e_i)=x\e_{3-i}$. Let $\pi$ be either the identity or the transposition map on $\{1,2\}$. Then similar computations give:
	\begin{align*}
		&\theta(\e_i)=x\e_{\pi(i)},	&&\theta(A^{q^{e_0-c}})=\left[\left(1-a_{\alpha}^{q^{e_0-c}}\right)\eta\textbf{f},A^{q^{e_0-c}}\right],\\	&\theta(T)=\left[(1-g)\eta\e_2,T^{-1}A^{q^{e_0-1}}\right],	&&\theta(B^{s/d})=\left[\left(1-a_{\beta}^{s/d}\right)\eta\textbf{f},B^{s/d}\right],
	\end{align*}
	for $0 \leq \eta \leq p-1$. This again gives $p(p-1)$ choices for $\theta$.
	
	In total, we have $|\Aut(M,M')|=2p(p-1)$.
\end{proof}
We now turn to looking at $|\Aut(M,M')|$ for the transitive subgroups in row 2 of Table \ref{metabelian-trans-subgroups}. We denote by $M_{u,c}$ the groups $P \rtimes \langle TA^{uq^{e_0-c}}\rangle$ which correspond to $d=1$, and by $\widehat{M}_{u,c}$ the groups $P \rtimes \langle TA^{uq^{e_0-c}},B^{s/d}\rangle$ with $d>1$. We recall that $1 \leq c \leq e_0$ and $u \in \mathbb{Z}_{q^c}^{\times}$.
\begin{proposition}
	The groups $M_{u,c}$ and $M_{u,c'}$ are isomorphic as permutation groups if and only if $c=c'$ and one of the following hold:
	\begin{enumerate}[label=\emph{(\Roman*)}]
		\item $u \equiv u' \pmod{q}$, or
		\item $c=1$ and $u+u'+1 \equiv 0 \pmod{q}$, or
		\item $c>1$ and $u \equiv -u' \pmod{q}$.
	\end{enumerate}	
	There are no further isomorphisms. Hence, the groups $M_{u,1}$ fall into  $(q-3)/2$ isomorphism classes of size $2$ for $u\neq -1,(q-1)/2$, and one isomorphism class of size $1$ for $u=(q-1)2$. For $c>1$, the groups $M_{u,c}$ fall into $q^{c-2}(q-1)/2$ isomorphism classes of size $2q$. Similarly for $\widehat{M}_{u,c}$.
	\end{proposition}
We remark that the groups $M_{-1,1}$ and $\widehat{M}_{-1,1}$ have been covered in Proposition \ref{metab_holgp_isoms}. They are both in isomorphism classes of size $2$.
\begin{proof}
	We ask when the groups
	\[M_{u,c}=P\rtimes \left\langle TA^{uq^{e_0-c}} \right\rangle\]
	are isomorphic as permutation groups for different choices of $u$ and $c$. 
	Firstly, we note that different choices of $c$ give different orders of $M_{u,c}$, therefore if $M_{u,c} \cong M_{u',c'}$ then $c=c'$. Now note that any element of $M_{u,c}$ of order $q^c$ has the form \[\left[\textbf{v},\left(TA^{uq^{e_0-c}}\right)^x\right]\]
	for some $\textbf{v}\in P$ and $x \in \mathbb{Z}_{q^c}^{\times}$, and acts on $P$ with two distinct eigenvalues, namely $g^xa_{\alpha}^{xuq^{e_0-c}}$ and $a_{\alpha}^{xuq^{e_0-c}}$. If, therefore, we have a choice of $u,u'$ such that $M_{u,c}\cong M_{u',c}$, then we must have
	\[\left\{g^xa_{\alpha}^{xuq^{e_0-c}},a_{\alpha}^{xuq^{e_0-c}}\right\}=\left\{ga_{\alpha}^{u'q^{e_0-c}},a_{\alpha}^{u'q^{e_0-c}}\right\}\]
	for some $x$. This is equivalent to solving the following equations $\bmod{ \;q^{e_0}}$:
	\begin{align}
		&x(q^{e_0-1}+uq^{e_0-c})\equiv q^{e_0-1}+u'q^{e_0-c}, &&\text{ and } \;\;\;
		xuq^{e_0-c}\equiv u'q^{e_0-c}, \text{ or} \label{simul_1}\\
		&x(q^{e_0-1}+uq^{e_0-c})\equiv u'q^{e_0-c}, &&\text{ and } \;\;\;
		xuq^{e_0-c}\equiv q^{e_0-1}+u'q^{e_0-c} \label{simul_2}.
	\end{align}
	We now turn to solving these systems.
	
	Consider first the pair of equations given by (\ref{simul_1}). We can rewrite by dividing through both equations by $q^{e_0-c}$ to get:
	\begin{align}
		&x(q^{c-1}+u) \equiv q^{c-1}+u' \mod{q^c},\label{firstsimul_1}\\
		&xu \equiv u' \mod{q^c}. \label{firstsimul_2}
	\end{align}
	We may then substitute (\ref{firstsimul_2}) into (\ref{firstsimul_1}) to obtain
	\[
		xq^{c-1}+u' \equiv q^{c-1}+u' \mod{q^c},
	\]
	and so $q^{c-1}(x-1) \equiv 0 \mod{q^c}$. This may be reduced to an equivalence modulo $q$:
	\[
		x \equiv 1 \mod{q}.
	\]
	Substituting $x=1+mq$ into the first equation in (\ref{simul_1}) gives
	\[(1+mq)u \equiv xu \equiv u' \mod{q^c},\]
	which is exactly equation (\ref{firstsimul_2}). Given therefore that $x \equiv 1 \mod{q}$, along with (\ref{firstsimul_2}), we may reduce down again to an equivalence modulo $q$ to obtain
	\[u \equiv u' \mod{q}.\]
	This gives us (I) in the proposition. We now look at the second pair of equation given by (\ref{simul_2}) in slightly less detail. Note that we may again rewrite both of them modulo $q^c$:
	\begin{align}
		&x(q^{c-1}+u) \equiv u' \mod{q^c},\label{secondsimul_1}\\
		&xu \equiv q^{c-1} + u' \mod{q^c}. \label{secondsimul_2}
	\end{align}
	Substituting (\ref{secondsimul_2}) into (\ref{secondsimul_1}) and reducing to an equivalence modulo $q$ gives
	\[x \equiv -1 \mod{q}.\]
	Given this, equation (\ref{secondsimul_1}) reduces to (\ref{secondsimul_2}), so we have just this equation to work with. We may again reduce to an equivalence modulo $q$, but in doing so we must take care and consider the cases $c=1$ and $c>1$ separately. If $c=1$, then (\ref{secondsimul_2}) becomes
	\[-u \equiv 1+u' \mod{q}.\]
	This gives us (II) in the statement of the proposition. If $c>1$ then (\ref{secondsimul_2}) becomes
	\[-u \equiv u' \mod{q}.\]
	This gives us (III).

	The case for $\widehat{M}_{u,c}$ is treated similarly. Note that, when $u\equiv u'\pmod{q}$, we must have $x \equiv 1\pmod{q}$, and when $u\equiv-u'\pmod{q}$, we must have $x \equiv -1 \pmod{q}$.
\end{proof}

\begin{proposition}
	For the groups in the second row of Table \ref{metabelian-trans-subgroups}, we have:
	\[\begin{array}{|c|c|c|}
		\hline
		\text{Parameters}	&	M	&	|\Aut(M,M')|\\
		\hline
		\text{all }u\neq -1,(q-1)/2	&	P \rtimes \langle TA^{uq^{e_0-1}}\rangle		 & p^2(p-1)\\
		\hline
		(u=-1)				&	P \rtimes \langle TA^{-q^{e_0-1}}\rangle		 & p(p-1)\\
		\hline
		u=(q-1)/2			&	P \rtimes \langle TA^{uq^{e_0-1}}\rangle		 & 2p^2(p-1)\\
		\hline
		d>1,\text{all }u\neq (q-1)/2		&	P \rtimes \langle TA^{uq^{e_0-1}},B^{s/d}\rangle & p(p-1)\\
		\hline
		d>1,u=(q-1)/2			&	P \rtimes \langle TA^{uq^{e_0-1}},B^{s/d}\rangle & 2p(p-1)\\
		\hline
		c,d>1, \text{any }u	&	P \rtimes \langle TA^{uq^{e_0-c}},B^{s/d}\rangle & p(p-1)\\
		\hline
	\end{array}\]
\end{proposition}\cblack
\begin{proof}
	We treat the cases for $M=M_{u,c}$ and $M=\widehat{M}_{u,c}$ together by setting $d=1$ and $d>1$ respectively in the following argument. We have $M'=\langle \textbf{f},A^{uq^{e_0-(c-1)}},B^{s/d} \rangle$. As in Proposition \ref{metab_holgp_isoms}, for $\theta \in \Aut(M,M')$, we have either $\theta(\e_i)=x\e_i$ or $\theta(\e_i)=x\e_{3-i}$ for some $1\leq x \leq p-1$. Further, we have
	\begin{align}
		&\theta\left(TA^{uq^{e_0-c}}\right)=\left[\textbf{v},T^aA^{auq^{e_0-c}}\right],\label{theta_TA}\\
		&\theta\left(TA^{uq^{e_0-c}}\right)^q=\theta\left(A^{uq^{e_0-(c-1)}}\right)=\left[\mu\textbf{f},A^{auq^{e_0-(c-1)}}\right], \text{ whenever }c>1,\label{theta_TAstab}\\
		&\theta\left(B^{s/d}\right)=\left[\nu\textbf{f},B^{bs/d}\right],\label{theta_Bstab}
	\end{align}
	for some $a \in \mathbb{Z}_{q^c}^{\times}$, $b$ coprime to $d$ and $0\leq \mu,\nu \leq p-1$. As $B^{s/d}\e_i=a_{\beta}^{s/d}\e_iB^{s/d}$, we must have $b=1$. There are now a number of possibilities, determined by the image of $\theta$ of $\e_i$, whether $c=1$ or $c>1$, the value of $u$, and whether $d=1$ or $d>1$. For clarity, suppose $\textbf{v}=m\e_1+n\e_2$.
	
	First, let us treat the case $c=1$ with $\theta(\e_i)=x\e_i$. In this case, given $TA^{uq^{e_0-1}}\e_1=g^{1+u}\e_1TA^{uq^{e_0-1}}$ and $TA^{uq^{e_0-1}}\e_2=g^u\e_2TA^{uq^{e_0-1}}$, we must have $a=1$. If $u=-1$, we must have $m=0$ in order for $\left[\textbf{v},TA^{-q^{e_0-1}}\right]$ to have order $q$. Other values of $u$ give no restrictions on $\textbf{v}$. If $d>1$, then as $B^{s/d}$ and $TA^{uq^{e_0-c}}$ commute, we have
	\begin{align}
		\label{beta_reltions}
		\begin{split}
			&m(1-a_{\beta}^{s/d})=\nu(1-g^{1+u}),\\
			&n(1-a_{\beta}^{s/d})=\nu(-1+g^u).
		\end{split}
	\end{align}
	If $d=1$, there are no further restrictions.
	
	Now consider $c=1$ with $\theta(\e_i)=x\e_{3-i}$. In this case, we must have $g^{a(1+u)}=g^u$ and $g^{au}=g^{1+u}$, which is only possible when $a=-1$ and $u=(q-1)/2$. If $d>1$, then we once again obtain (\ref{beta_reltions}), otherwise there are no further restrictions. This then retrieves the counts for $|\Aut(M,M')|$ given in the statement of the proposition, corresponding to $c=1$.
	
	Now suppose $c>1$. Note firstly that $\left[\textbf{v},T^aA^{auq^{e_0-c}}\right]$ has order $q^c$ regardless of the choice of $\textbf{v}$. Secondly, given the restriction (\ref{theta_TAstab}), a choice for one of $m,n$, and $\mu$ automatically determines a choice for the other two. 
	
	Suppose that $\theta(\e_i)=x\e_i$. Then we again have $a=1$. If $d=1$, there are no further relations. If $d>1$, we have
	\begin{align*}
		\label{beta_reltions2}
		\begin{split}
			&m(1-a_{\beta}^{s/d})=\nu(1-ga_{\alpha}^{uq^{e_0-c}}),\\
			&n(1-a_{\beta}^{s/d})=\nu(-1+a_{\alpha}^{uq^{e_0-c}}).
		\end{split}
	\end{align*}
	Thus a choice of one of either $m,n$ or $\nu$ determines the value for the other two (and also $\mu$ as above).
	
	Finally, suppose that $\theta(\e_i)=x\e_{3-i}$. In this case, we must have $a_{\alpha}^{auq^{e_0-c}}=ga_{\alpha}^{uq^{e_0-c}}$ and $g^aa_{\alpha}^{auq^{e_0-c}}=a_{\alpha}^{uq^{e_0-c}}$. This system has no solution for $c>1$, and so there are no such elements of $|\Aut(M,M')|$ satisfying this property.
	
	We therefore have:
	\[|\Aut(M_{u,c},M'_{u,c})|=|\Aut(\widehat{M}_{u,c},\widehat{M}'_{u,c})|=p(p-1) \;\; \text{ for } c>1.\]
\end{proof}

\begin{remark}
	\emph{Note that the groups $M_{u,c}$ and $\widehat{M}_{u,c}$ contain an abelian group of order divisible by $pq$ if and only if $(u,c)=(-1,1)$. In this case, the generator $TA^{-q^{e_0-c}}$ commutes with $\e_1$, so $M_{-1,1} \cong C_p \times (C_p \rtimes C_q)$ and $\widehat{M}_{u,c} \cong (C_p \times (C_p \rtimes C_q)) \rtimes C_d$. Recall that $M_{-1,1}$ is also isomorphic as a permutation group to $P \rtimes \langle T \rangle$, and $\widehat{M}_{-1,1}$ is isomorphic as a permutation group to $P \rtimes \langle T,B^{s/d} \rangle$. Each contains a normal subgroup of order $pq$ which is complement to $M_{-1,1}'$ or $\widehat{M}'_{-1,1}$ respectively. For $(u,c) \neq (-1,1)$, the generator $TA^{uq^{e_0-c}}$ acts on $P$ with two distinct eigenvalues as discussed above. The normal complement in $M_{u,c}$ to $M'_{u,c}$ (and likewise in $\widehat{M}_{u,c}$ to $\widehat{M}'_{u,c}$) would be a transitive subgroup of order $pq$, namely either $\langle \e_1,[\mu\e_2,TA^{uq^{e_0-1}}] \rangle$ or $\langle \e_2, [\mu\e_1,TA^{uq^{e_0-1}}] \rangle$ for some $0 \leq \mu \leq p-1$, $u \neq -1$. However, none of these groups are normalised by $P$, and so $M_{u,c}'$ does not have a normal complement in $M_{u,c}$.}
	
	\emph{For $(u,c) \neq (-1,1)$, we denote the isomorphism classes of the groups $M_{u,c}$ and $\widehat{M}_{u,c}$ by $\mathbb{F}_p^2 \rtimes_u C_{q^c}$ and $\mathbb{F}_p^2 \rtimes_u C_{dq^{c}}$ respectively.}
\end{remark}

	

\begin{proposition}
	For a fixed $d|s$ with $M= \langle \e_1,T,B^{s/d} \rangle$, we have
	\[|\Aut(M,M')|= \begin{cases}
		p(p-1) & \text{if $d = 1$},\\
		p-1 & \text{otherwise}.
	\end{cases}\]
\end{proposition}
\begin{proof}
	Clearly $M'=\langle B^{s/d} \rangle$. For $\theta \in \Aut(M,M')$, we must have
	$\theta(\e_1)=x\e_1$ for some $1 \leq x \leq p-1$, $\theta(B^{s/d})=B^{as/d}$ for some $a$, and $\theta(T)=[\mu\e_1,T^b]$ for some $\mu$ and $b$. Suppose first that $d\neq 1$, then the relations $B^{s/d}\e_1B^{-s/d}=a_{\beta}^{s/d}\e_1$, $T\e_1T^{-1}=g\e_1$ and $TB^{s/d}=B^{s/d}T$ force $a=b=1$ and $\mu=0$, giving $|\Aut(M,M')|=p-1$. If $d=1$, then there is no restriction on $\mu$, giving $|\Aut(M,M')|=p(p-1)$.
\end{proof}

\begin{proposition}
	For the groups in row four of Table \ref{metabelian-trans-subgroups} which are not isomorphic to $\langle \e_1,T,B^{s/d} \rangle$, we have:
	\[\begin{array}{|c|c|c|}
		\hline
		\text{Parameters}	&	M	&	|\Aut(M,M')|\\
		\hline
		(u,c)=(-1,1)		&	\left\langle \e_1,TA^{-q^{e_0-1}},B^{s/d}\right\rangle	&	(p-1)(q-1)\\
		\hline
		c>1					&	\left\langle \e_1,TA^{uq^{e_0-c}},B^{s/d}\right\rangle	&	p-1\\
		\hline
	\end{array}\]
\end{proposition}
\begin{proof}
	We have $M'=\langle A^{uq^{e_0-(c-1)}},B^{s/d} \rangle$. Then $\theta(\e_1)=x\e_1$ and $\theta(B^{s/d})=B^{s/d}$ by similar logic to the previous proof. Further, we have
	\[\theta(TA^{uq^{e_0-c}})=\left[\mu\e_1,T^aA^{auq^{e_0-c}}\right] \text{ with } \theta(A^{uq^{e_0-(c-1)}})=A^{auq^{e_0-(c-1)}}.\]
	Firstly, if $c>1$, then given $TA^{uq^{e_0-c}}\e_1=ga_{\alpha}^{uq^{e_0-c}}\e_1TA^{uq^{e_0-c}}$, we must have $\mu=0$ and $a=1$. There is no restriction on $x$, and so $|\Aut(M,M')|=p-1$. If $c=1$ but $u\neq -1$, then $M$ is isomorphic to $\langle \e_1,T,B^{s/d} \rangle$. So suppose $(c,u)=(1,-1)$, then $M$ contains an abelian subgroup of order $pq$ coprime to its index of $2$. Then we clearly have $|\Aut(M,M')|=(p-1)(q-1)$, corresponding to $\mu=0$, $1 \leq x \leq p-1$ and $1 \leq a \leq q-1$.
\end{proof}

\begin{proposition}
	For fixed $c\geq 1$ and $d \mid s$ with $M=\left\langle \e_1,T,A^{q^{e_0-c}},B^{s/d}\right\rangle$, we have $|\Aut(M,M')|=(p-1)(q-1)$.
\end{proposition}

\begin{proof}
 We have $M'=\langle A^{q^{e_0-c}},B^{s/d} \rangle$. Let $\theta \in \Aut(M,M')$, then as $\e_1$ is an eigenvector for each element of $M'$, we must have that $\theta|_{M'}=\text{id}$. We then have $\theta(\e_1)=x\e_1$ and
 \[\theta(T)=\left[\mu\e_1,T^aA^{bq^{e_0-1}}\right],\]
 for some $1 \leq x \leq p-1$, $0\leq \mu \leq p-1$, $1 \leq a \leq q-1$, and $0 \leq b \leq q-1$. Given $c\geq 1$, as $T$ and $A^{e^{e_0-c}}$ commute, we must have $\mu=0$. Next, as $T\e_1=g\e_1T$, we have $a+b=1$. There are no further restrictions, giving a total of $(p-1)(q-1)$ elements in $\Aut(M,M')$.
\end{proof}

\begin{table}
	\setlength{\extrarowheight}{2mm}
	\centering
	\begin{tabular}{|c|c|}
		\hline
		Group	&	Structure\\
		\hline
		$P \rtimes \left \langle T, A^{q^{e_0-c}}, B^{s/d} \right \rangle$ & $(N \rtimes (C_p \rtimes C_{q^{c}}))\rtimes C_d$\\[5pt]
		\hline
		$P \rtimes \left \langle TA^{uq^{e_0-c}}, B^{s/d} \right \rangle$, for $(c,u) \neq (1,-1)$ & $\mathbb{F}_p^2 \rtimes_u C_{dq^{c}}$\\[5pt]
		\hline
		$P \rtimes \left \langle T, B^{s/d} \right \rangle \cong P \rtimes \left\langle TA^{-q^{e_0-1}} \right \rangle$ & $((C_p \rtimes C_q)\times C_p)\rtimes C_d$\\[5pt]
		\hline
		$\left\langle \e_1,T,B^{s/d} \right\rangle$ & $C_p \rtimes C_{dq}$\\[5pt]
		\hline
		$\left\langle \e_1, TA^{uq^{e_0-c}}, B^{s/d} \right\rangle$, for $(c,u) \neq (1,-1)$ & $C_p \rtimes C_{dq^{c}}$\\[5pt]
		\hline
		$\left\langle \e_1, TA^{-q^{e_0-1}}, B^{s/d} \right\rangle$ & $(C_p \rtimes C_d) \times C_q$\\[5pt]
		\hline
		$\left\langle \e_1,T,A^{q^{e_0-c}}, B^{s/d} \right\rangle$ & $(C_p \rtimes C_{dq^c}) \times C_q$\\[5pt]
		\hline
		$\left\langle \e_1,T \right\rangle$ & $C_p \rtimes C_q$\\[5pt]
		\hline
		$\left\langle \e_1, TA^{uq^{e_0-c}} \right\rangle$, for $(c,u) \neq (1,-1)$ & $C_p \rtimes C_{q^c}$\\[5pt]
		\hline
		$\left\langle \e_1, TA^{-q^{e_0-1}} \right\rangle$ & $C_{pq}$\\[5pt]
		\hline
		$\left\langle \e_1,T,A^{q^{e_0-c}} \right\rangle$ & $(C_p \rtimes C_{q^c}) \times C_q$\\[5pt]
		\hline
	\end{tabular}
	
	\medskip
	
	\caption{Isomorphism types of transitive groups for $N$ metabelian.}
	\label{metabelian-structure}
\end{table}

Most notably there is no case where the groups are isomorphic as abstract groups but not isomorphic as permutation groups.

Denote by $\sigma_0(s)$ the number of divisors of $s$.

\begin{table}
	\setlength{\extrarowheight}{2mm}
	\begin{center}
		\resizebox{\columnwidth}{!}{
		\begin{tabular}{|c|c|c|c|}
			\hline
			Structure	&	$\#$groups	&	$|\Aut(M,M')|$	&	$\#$HGS \\
			\hline
			$(N \rtimes (C_p \rtimes C_{q^{c}}))\rtimes C_d$ $c>0, d|s$ 	&	$1$	&	$2p(p-1)$	&	$2$ \\[5pt]
			\hline
			$(C_p \rtimes C_q) \times C_p$	&	$2$	&	$p(p-1)$	&	$2$ \\
			$\mathbb{F}_p^2 \rtimes_u C_q$, $1 \leq u \leq (q-3)/2 $	&	$2$	&	$p^2(p-1)$	&	$2p$\\
			$\mathbb{F}_p^2 \rtimes_{(q-1)/2}C_q$	&	$1$	&	$2p^2(p-1)$	&	$2p$\\
			$\mathbb{F}_p^2 \rtimes_u C_{q^c}$, $c>1, 1\leq u \leq (q-1)/2$	&	$2q^{c-1}$	&	$p(p-1)$		&	$2q^{c-1}$\\
			[5pt]
			\hline
			$(C_p \times (C_p \rtimes C_q)) \rtimes C_d$, $d|s, d > 1$	&	$2$	&	$p(p-1)$	&	$2$\\
			$\mathbb{F}_p^2 \rtimes_u C_{dq}$, $1\leq u \leq (q-3)/2,d|s, d>1$	&	$2$	&	$p(p-1)$	&	$2$\\
			$\mathbb{F}_p^2 \rtimes_{(q-1)/2}C_{dq}$, $d|s, d>1$	& $1$ &	$2p(p-1)$	&	$2$ \\
			$\mathbb{F}_p^2 \rtimes_u C_{dq^c}$, $c>1, 1\leq u \leq (q-1)/2,d|s, d>1$	& $2q^{c-1}$ &	$p(p-1)$	& $2q^{c-1}$\\[5pt]
			\hline
			$C_p \rtimes C_{dq^c}$, $c>0,d|s, (c,d) \neq (1,1)$	&	$2p\varphi(q^{c})$	&	$p-1$	&	$2\varphi(q^c)$	\\[5pt]
			\hline
			$C_p \rtimes C_q$	&	$2p(q-1)+2$	&	$p(p-1)$	&	$2p(q-2)+2$	\\[5pt]
			\hline
			$(C_p \rtimes C_{dq^c}) \times C_q$, $c\geq0, d|s$	&		$2p$	&	$(p-1)(q-1)$	&	$2(q-1)$	\\[5pt]
			\hline
		\end{tabular}}
		
		\medskip
		
		\caption{\# HGS Transitive subgroups for $N$ metabelian.}
		\label{metabelian-trans-HGS}
	\end{center}
\end{table}

\begin{theorem}
	In total, there are
	\[\sigma_0(s)\left[3e_0+3+\frac{1}{2}(q-3)+\frac{1}{2}(e_0-1)(q-1)\right]\]
	isomorphism types of permutation groups $G$ of degree $pq$ (as listed in Table \ref{metabelian-structure}) which are realised by Hopf-Galois structures of non-abelian type $C_p \rtimes C_q$ (the numbers of which are displayed in Table \ref{metabelian-trans-HGS}). These include the two regular groups, i.e. the cyclic and non-abelian groups of order $pq$ (for which the corresponding Galois extensions have $2(q-1)$ and $2p(q-2)+2)$ Hopf-Galois structures of non-abelian type respectively). Of the corresponding field extensions for these groups, $\sigma_0(s)(1+\frac{1}{2}(q-3)+\frac{1}{2}(e_0-1)(q-1))$ fail to be almost classically Galois (these are the groups given by all but one group in each of the second and third left-hand cells in Table \ref{metabelian-trans-HGS}, along with $C_p \rtimes C_{dq^{c}}$ for $c>1$). In the remaining $\sigma_0(s)(3e_0+2)$ cases, the extensions are almost classically Galois.
\end{theorem}
\begin{proof}
	Everything except the statements about almost classically Galois extensions follows from Table \ref{metabelian-trans-HGS}. Let $J$ be a group of order $pq$ in either $P\rtimes\left\langle TA^{uq^{e_0-c}},B^{s/d} \right\rangle$ or $\left\langle \e_1,TA^{uq^{e_0-c}},B^{s/d} \right\rangle$ for $(u,c)\neq (-1,1)$. It is clear that $J$ cannot contain $T$, and so none of these groups contain a regular normal subgroup. Therefore neither families of groups correspond to almost-classically Galois extensions. For the rest of the groups in the table, it is clear that they contain $\langle \e_1,T \rangle \cong N$, and so correspond to almost-classically Galois extensions.
\end{proof}

\begin{remark}\label{metab_BML}
	\emph{Note that specialising to $p=2q+1$, we obtain $2(3+3+\frac{1}{2}(q-3))=q+9$ isomorphism types, with $2(1+\frac{1}{2}(q-3)+\frac{1}{2}(1-1)(q-1))=q-1$ non-almost classically Galois extensions and $2(3+2)=10$ almost classically Galois extensions. This retrieves the result of Theorem 3.2 of \cite{BML21}. We also note that the result for $(C_p \rtimes C_q) \times C_p$ where $|\Aut(M,M')|=p(p-1)$ corrects an error in Table $5$ in \cite{BML21}.}
\end{remark}

\begin{theorem}
	There are $\sigma_0(s)(2e_0+1)$ permutation groups $G$ of degree $pq$ which are realised by Hopf-Galois structures of cyclic and non-abelian types. These are shown in Table \ref{both-types}. In all but $\sigma_0(s)(e_0-1)$ of these cases (corresponding to $C_p \rtimes C_{dq^{c}}$ for $c>1$), the corresponding field extensions are almost classically Galois.
\end{theorem}

\begin{remark}
	\emph{For every Hopf-Galois structure of non-abelian type, there is an opposite Hopf-Galois structure of the same type. Hence it makes sense that the number of Hopf Galois structures in each row of Table \ref{metabelian-trans-HGS} is even.}
\end{remark}

\begin{remark}\label{both_BML}
	\emph{This again realises the result of \cite{BML21}, that there are $2(2+1)=6$ permutation groups of degree $pq$ with $p=2q+1$ which are realised by Hopf-Galois structures of both cyclic and metabelian type (Theorem 3.3 of \cite{BML21}). In this case, $2(1-1)=0$ extensions fail to be almost classically Galois.}
\end{remark}

\begin{table}
	\begin{center}
		\begin{tabular}{|c|c|c|}
			\hline
			Structure	&	$\#$ cyclic type HGS	&	$\#$ non-abelian type HGS\\
			\hline
			$(C_p \rtimes C_{dq^{c}}) \times C_q$	&	$1$	&	$2(q-1)$	\\
			\hline
			$C_p \rtimes C_{dq^{c}}$, $(c,d) \neq (1,1),(0,d)$	&	$1$	&	$2\varphi(q^{c})$	\\
			\hline
			$C_p \rtimes C_q$	&	$p$	&	$2p(q-2)+2$\\
			\hline
		\end{tabular}
		\medskip
		\caption{Groups admitting Hopf-Galois structures of both types.}
		\label{both-types}
	\end{center}
\end{table}


\section*{Data availability}
The research data supporting this publication are provided within this paper.

\section*{Acknowledgements}
The author is supported by the Engineering and Physical Sciences Doctoral Training Partnership research grant EP/T518049/1 (EPSRC DTP).

\bibliography{MyBib}

\end{document}